\documentclass[a4paper,12pt,reqno]{amsart}
\usepackage{latexsym}
\usepackage{amssymb} 
\usepackage{mathrsfs}
\usepackage{amsmath}
\usepackage{latexsym}
\usepackage{delarray}
\usepackage{amssymb,amsmath,amsfonts,amsthm,mathrsfs}
\usepackage[utf8x]{inputenc}

 \newtheorem{thm}{Theorem}[section]
 \newtheorem{cor}[thm]{Corollary}
 
 \newtheorem{prop}[thm]{Proposition}
 \theoremstyle{definition}
 \newtheorem{defn}[thm]{Definition}
 \theoremstyle{remark}
 \newtheorem{rem}[thm]{Remark}
 \newtheorem*{ex}{Example}
 \newtheorem*{exs}{Examples}
 \numberwithin{equation}{section}

\newcommand{\bi}{\begin{itemize}}
\newcommand{\ei}{\end{itemize}}
\newcommand{\be}{\begin{enumerate}}
\newcommand{\ee}{\end{enumerate}}
\newcommand{\beq}{\begin{equation}}
\newcommand{\eq}{\end{equation}}

\begin{document}
%
%
%
%
%
%
%
%
%
\title[G\^ateaux differentiability on non-separable Banach spaces]
 {G\^ateaux differentiability on non-separable Banach spaces}
\author[Andrés F. Mu\~noz-Tello]{Andrés F. Mu\~noz-Tello}

\address{
Facultad de Ciencias Basicas\\
Universidad Santiago de Cali\\
Calle 5,\, 62-00, Pampalinda, Santiago de Cali\\
Colombia
}

\email{andres.munoz00@usc.edu.co}




\date{\today}
\dedicatory{Dedicated to the memory of Professor Guillermo Restrepo}
\begin{abstract}
 This paper deals with the extension of a classical theorem  by R. Phelps on the G\^ateaux differentiability of Lipschitz functions on separable Banach spaces  to the non-separable case. The extension of the theorem is not possible for general  non-separable Banach spaces. Therefore, we consider a norm of the non-separable space $\ell^{\infty}(\mathbb{R})$, showing that it complies with the aforementioned theorem thesis. We also consider other examples as $ L^{\infty}(\mathbb{R}) $ and $NBV[a, b]$, showing in all cases their sets of $G$-differentiability and some properties of these sets. All of the above, closely following the assumptions in Phelps Theorem, which will extend in the case separable to functions with rank over the Asplund spaces and in the non-separable case for projective limits and locally-Lipschitz cylindrical functions.
\end{abstract}
\maketitle

\section{Introduction}
The Phelps's Theorem~\cite{ph}, which will be extended in this paper, shows  that any locally-Lipschitz function of a Banach space separable to another one with Radon-Nikodým Property, is G\^ateaux differentiable outside of a Gaussian null set; thus ge\-ne\-ra\-li\-zing Rademacher Theorem ~\cite{ra} of continuous functions on $\mathbb{R}^{n}$.\\

The Phelps's Theorem is not only of interest in functional analysis; in recent articles, this theorem is applied to the study of diffusion equations, elliptic equations in infinite dimension and stochastic control, as does Goodys~\cite{gol}, which uses the Phelps Theorem to make the necessary construction to solve the Hamilton-Jacobi-Bellman equation related to the optimal control of a stochastic semilinear equation over a Hilbert space $H$. On the other hand, this theorem is useful in the geometry of Banach spaces, on convex functions and S\'obolev spaces in infinite dimensional domains, as it is quoted  in Chapter 5 of the book of Evans~\cite{ev}. In addition, having served as a starting point and guide in the study of new variational methods applied to the Fr\'echet differentiability of functions locally Lipschitz as raised by Lindenstrauss {\it et al.}~\cite{lin}. \\

In the case that $X$, $Y$ are two Banach spaces and $\Omega\subseteq X$ an open set, a function $f:\Omega\rightarrow Y$ is {\it G\^ateaux differentiable ($G$-differentiable)} in $a\in\Omega$, if there is a continuous linear function $u$ such that for all $h\in X$ the limit
$$\lim\limits_{t\rightarrow 0^{+}} \frac{f(a+th)-f(a)}{t}=u \left(h\right)$$
exists, with respect to the norm defined in the topology.\\

On normed vector spaces, particularly on Banach spaces $X$, $Y$ a function $f: X \rightarrow Y$ is {\it Fr\'echet differentiable ($F$-differentiable)} in $a\in X$, if there is a continuous linear function $u$ such that
$$f\left(a+h\right)-f\left(a\right)-u\left(h\right)=r\left(h\right)$$ where $\lim\limits_{h\rightarrow 0} \frac{\left\|r\left(h\right)\right\|_{Y}}{\left\| h \right\|_{X}}=0$.\\

In the case of two Banach spaces $X$, $Y$ it is said that $f:X \rightarrow Y$ is {\it Hadamard differentiable ($H$-differentiable)} in $ a\in X $, if and only if there exists $ u\in\mathcal{L}\left(X, Y \right)$ such that
\begin{equation}
u\left(h\right) = \lim\limits_{\left(t,k\right)\rightarrow \left(0^{+}, h\right)}\frac{f\left(a + tk\right) - f\left(a\right)}{t}\ \ \textnormal{for all}\ h\in X.\label{$H$-derivada}\nonumber
\end{equation}
 
In Schirotzek~\cite{Sch}, using the definition and properties of bornologies, it is demonstrated that to be $F$-differentiable implies to be $H$-differentiable and in turn to be $H$-diffe\-ren\-tia\-ble implies to be $G$-differentiable, showing that in spaces of finite dimension, the derivative of Hadamard coincides with the derivative of Fréchet.\\

The following relations and properties will be very useful in the development of this document and their proofs can be found in Proposition 2.51 and Theorem 2.28, from Penots~\cite{pe} and in Coleman Theorem 2.1~\cite{col}.
	
	\begin{prop}\label{b1} Let $E$, $F$ and $G$ Banach spaces, $\Omega\subseteq E$, $\Omega'\subseteq F$ open sets and $f:\Omega \rightarrow F$, $g:\Omega'\rightarrow Z$ are functions such that $f \left(\Omega \right) \subseteq \Omega'$, where $ b = f \left(a \right)$.
	\begin{enumerate}
		\item If $f$ is $G$-differentiable in $ a \in \Omega$, and
		locally Lipschitz in $a\in\Omega $, then $H$ -differentiable and $d_{H} f\left(a\right)=d_{G} f \left (a \right)$.
	
	\item If $f$ is a function continuously $G$-differentiable, then $f$ is $H$-differentiable.
	
	\item If $f$ is $H$-differentiable in $a\in\Omega$ and $g$ is $H$-differentiable in $b \in\Omega'$, then $g\circ f:\Omega\rightarrow Z$ is $H$-differentiable in $a \in \Omega$ and
	\[d_{H}\left(g\circ f\right)\left(a\right) := d_{H} g\left(b\right)\circ d_{H} f\left(a\right).\]
	
	\item If $f$ is $F$-differentiable in $a\in\Omega$ and $g$ is $F$ -differentiable in $b\in\Omega'$, then $g\circ f:\Omega\rightarrow Z$ is $F$-differentiable in $a\in\Omega$ and
	\[d_{F}\left(g\circ f\right)\left(a\right) := d_{F} g\left(b\right)\circ d_{F} f\left(a\right).\]
		\end{enumerate}
	\end{prop}
	
On the other hand, for a topological space$(X, \tau)$, a set is called {\it $G_{\delta}$ set} if it can be expressed as a countable intersection of open sets. This type of set is fundamental in the definition of the spaces {\it F and $G$-Asplund}, since these respectively are Banach spaces in which all convex and continuous function defined in an open and convex subset is $F$-differentiable ($G$-differentiable) in a dense and $G_{\delta}$ set. The following are some examples, taken from Phelps~\cite{ph} and Asplund~\cite{as}, of the spaces in question.\\

\begin{exs}\label{a4}\
	\begin{enumerate}
	\item $\mathbb{R}^{n}$ is a $F$-Asplund space.
	
	\item\label{a5} All separable Banach space is a $G$-Asplund space.
	\end{enumerate}
	 \end{exs}
	
Recall that the set of all $F$-Asplund spaces is strictly contained in the set of all $G$-Asplund spaces. In fact, from the previous example $\ell^{1}(\mathbb{R})$ is a $G$-Asplund space and of the demonstrated in Deville~\cite{de}, the usual norm of $\ell^{1}(\mathbb{R})$ is not $F$-differentiable in any point, and therefore $\ell^{1}(\mathbb{R})$ is not a $F$-Asplund space.\\

In addition, $E$ is an $F$-Asplund space if and only if $E'$ has the RNP, where the RNP refers to that $E$ has the {\it Radon-Nikodým Property} with respect to any finite measurement space $(\Omega, \mathcal{F}, \mu)$. That is, if every  $\sigma$-aditive function, of bounded variation and $\mu$-continuous $\alpha:\mathcal{F}\rightarrow E$ admits an integral representation of the form
		$$\alpha(A)=\int_{A}g(\omega)d\mu(\omega), $$
for all $A\in\mathcal{F}$ and $g\in L_{1}(\mu,E)$. The following examples, which appear in the works of Namioka~\cite{na} and Stegall~\cite{st}, illustrate the above. 
		
\begin{exs}\
	\begin{enumerate}
	\item Every finite dimensional Banach space  has the RNP.
	\item The Banach space $c_{0}$ does not have the RNP.
	\item The Banach space $\ell^{1} (\mathbb{K})$ has the RNP.
	\item Every dual and separable Banach space has the RNP.
	\item Every reflexive Banach space has the RNP.
	\end{enumerate}
	\end{exs}

Now, some results will be shown about separable Banach spaces, in particular $\ell^{1}(\mathbb{R})$ and $\mathcal{C}(T)$ in which the Phelps Theorem is used, also demonstrating some extra pro\-per\-ties of the sets of $G$-differentiability of the norms of each space. Later, we will work on norms of $ \ell ^{\infty}(\mathbb{R})$, $L^{\infty}(\mathbb{R})$ and $NBV[a, b]$, demonstrating that in any non-separable Banach space it is not possible to extend the aforementioned theorem. Finishing with an extension of the Phelps Theorem for projective boundaries and cylindrical functions locally Lipschitz.

\section[$G$-differentiable in separable spaces]{$G$-differentiable in separable spaces}
In the Proposition \ref{b1} it was shown that there was a relation between $G$-di\-ffe\-ren\-tia\-ble functions and $H$-differentiable functions via local Lipschitz functions, moreover  that a Banach space is $F$-Asplund if and only if its dual has the RNP. In this sense, first a result will be explained, which shows  that the Phepls Theorem does not only speak about  $G$-differentiability is actually a result that ensures  the $H$-differentiability. The proof only relies on  Proposition ~\ref{b1}. Later a variation of the Phelps Theorem in which the Asplund spaces are involved will be obtained.

\begin{cor}
	Let $X$ be a separable Banach space and $Y$ a Banach space satisfying the RNP. Let $\Omega\subset X$ a open set and non-empty. If $f:\Omega\rightarrow Y$ is locally Lipschitz function. Then $f$ is $H$-differentiable in the complement of a null Gaussian set.
	\end{cor}
	
	\begin{thm}\label{composicion00}
	Given $X$ a separable Banach space, $\Omega\subset X$ a open set and $Y'$ a dual space with the RNP, if $\psi:Y'\rightarrow Y$ is $H$-differentiable in all $Y'$ y $T:\Omega\rightarrow Y'$ is locally Lipschitz function. Then $Y$ is Asplund and there exists a function $\tau: \Omega\rightarrow Y$ that is $H$-differentiable except in a null Gaussian set.
	\end{thm}

\begin{proof}	
From the Phelps Theorem, $T$ is $G$-differentiable except in a set $B^{c}$ of null Gaussian measure ($\mu(B^{c})=0$). Since $T$ is locally Lipschitz then T is $H$-differentiable except in $B^{c}$ and if $\tau=\psi\circ T$ is defined then by the chain rule in the $H$-derivative, for $x\in B$, $y'\in Y'$, $ h\in X$ and $ h'\in Y'$ we have that:
	$$d_{H}\tau(x,h)=d_{H}\psi(y',h')\cdot d_{H}T(x,h).$$ 
	Therefore $\tau$ is $H$-differentiable in $B$ and 
{\setlength\arraycolsep{0pt}
	\begin{eqnarray}
	\mu\{x\in X &:& d_{H}\tau(x,h)\ exists\}\nonumber\\&=&\mu\{x\in X:d_{H}\psi(y',h')\cdot d_{H}T(x,h) \ exists\}\nonumber\\&=&\mu\{x\in X: d_{H}T(x,h) \ exists\}  =\mu(B). \nonumber
	\end{eqnarray}}
That is, $\mu \{x \in X: d_{H} \tau(x, h) \ does\ not \ exists \} = \mu (B^{c})$, completing the proof.
\end{proof}
	
	\begin{rem}
If $ Y=Y'$ it is clear that this theorem coincides with  Phelps Theorem.
	\end{rem}
Now, we will state a pair of examples, extracted from Deville~\cite{de}, where sets of $G$-differentiability of norms are shown in certain separable Banach spaces. In these examples the application of Phelps Theorem is allowed, since  the norms are Lipschitz functions and since $\mathbb{R}$ has the RNP.   
	
\begin{exs}\
\begin{enumerate}
\item Let $E=\ell^{1}(\mathbb{R})$ be the Banach space of the sequences $x=(x_{n})_{n\in\mathbb{N}} $ such that $\sum_{n}|x_{n}|<\infty $ and endowed with the norm $\varphi(x)=\sum_{n}|x_{n}|$. If $B\subseteq \ell^{1}(\mathbb{R})$ is the set of the sequences $(x_{n})_{n\in\mathbb{N}}$ such that $x_ {n}\neq 0$ for all $n\in\mathbb{N}$. Then the set of $G$-differentiability of $\varphi$ is $B$ and the $G$-derivative of $\varphi$ in $x$ is
$$d_{G}\varphi(x)=(sig(x_{n}))_{n\in\mathbb{N}}\quad \textnormal{and}\quad d_{G}\varphi(x,h)=\sum_{n}sig(x_{n})h_{n}.$$
\item Let $T$ be a Hausdorff and compact topological space. Then 
	\begin{itemize} 
	 \item[$(i)$] The norm $\varphi(x)=\sup_{t \in T}|x(t)|$ is $G$-differentiable at $x \in\mathcal{C}(T)$ if and only if $\varphi(x)$ reaches the supremum at a single point $t_ {0}\in T$ and has $G$-derivative $$\partial_{G}\varphi(x)(h)=sig((x(t_{0})),\quad\textrm{that is}\quad \partial_{G}\varphi(x)=sig(x(t_{0}))\delta_{t_{0}}.$$ 
	 \item[$(ii)$] The norm $\varphi(x)=\sup_{t\in T}|x(t)|$ is $F$-differentiable in $x \in\mathcal{C}(T)$ if and only if $\varphi(x)$ reaches the supreme in a single isolated point $t_ {0}\in T$ and has $F$-derivative $$\partial_{F}\varphi(x)=sig(x(t_{0}))\delta_{t_{0}}.$$
	\end{itemize}

\end{enumerate}

\end{exs}

From the previous example, it can also be shown that the set of $G$-differentiability of the norm $\varphi(\cdot)$ defined in both spaces is dense and $G _{\delta}$.   

\begin{thm}\
\begin{enumerate}
\item Given the norm $\varphi(x) = \sum_{n}\|x_{n}\| $ on the space $\ell^{1}(\mathbb{R})$. The $G$-differentiability set of $\varphi(\cdot)$, is $G _{\delta}$ and dense.
\item Given the norm $ \varphi(x) =\sup_{t\in T}| x(t)|$ on the space $\mathcal{C}(T)$.  The set of $G$-differentia\-bility of $\varphi(\cdot)$, is open and dense. 
\end{enumerate}
\end{thm}
\begin{proof}
${(1)}$ Since the linear form $x \rightarrow x_{n}$ is continuous, then the set $A_{n}$ of the $x$ such that $x_ {n}\neq 0$, is open for all $n$ and $B=\cap_{n} A_{n}$ is $G_{\delta}$. Now $B$ is dense in $\ell^{1}(\mathbb{R})$, since if you take $x\in\ell^{1}(\mathbb{R})$, $\epsilon >0$ and define $y_{n}=x_ {n}$ if $x_ {n}\neq0 $ and $y_{n}=\epsilon / 2^{n}$ if $x_ {n}=0$, then $\varphi(y-x)<\epsilon$.

\medskip\noindent${(2)}$  If $x\in T$, $f\in B$, $g \in B(f,\epsilon/2)$ and $\epsilon=|f(x_{0})| - |f(x)|$ then
\begin{eqnarray}
	|g(x)|+\epsilon&=&|g(x)|+|f(x_{0})|-|f(x)|\leq|g(x)-f(x)|+|f(x_{0})|\nonumber\\&<&\frac{\epsilon}{2}+|f(x_{0})|\leq \epsilon+|f(x_{0})-g(x_{0})|+|g(x_{0})|\nonumber\\&=&\frac{\epsilon}{2}+\frac{\epsilon}{2}+|g(x_{0})|=\epsilon+|g(x_{0})|\nonumber
	\end{eqnarray}
that is, $|g(x)|<|g(x_{0})|$. Therefore $B(f, \epsilon / 2)\subseteq B$ for all $f\in B$, meaning that $B$ is open, which is also $G_{\delta}$.

Let $ f\in B^{c}$ and $\epsilon> 0$. since $f$ is continuous over $\overline{B(x_{1}, \varepsilon / 2)}$, then $f$ is a bounded function and therefore there exists $x_{1} \in \overline{B(x_ {1}, \varepsilon / 2)}$ such that $\varphi(f) - \frac{\epsilon}{4}<|f(x_{1})| \leq \varphi(f)$. If the border points are represented as $\delta_{x_{1}}$ and the function $g$ as:		
\begin{center}
\begin{tabular}{lll}
	$g(x)= \left\{ \begin{array}{lll}
	\frac{f(x_{1})-f(\delta_{x_{1}})}{\epsilon/2}\|x\|+\frac{\|x_{1}\|f(\delta_{x_{1}})-f(x_{1})\|\delta_{x_{1}}\|+(\epsilon/2)^{2}}{\epsilon/2}   & \ \textrm{if $x\in B(x_{1},\varepsilon/2)$} \\
	f(x)  & \ \textrm{if $(B(x_{1},\varepsilon/2))^{c}$. }
	\end{array} \right.$
	\end{tabular} 
	\end{center}
then $g\in B$ and also
	$$\varphi(f-g)=\sup_{x\in\overline{B(x_{1},\varepsilon/2)}}|g(x)-f(x)|=\sup_{x\in B(x_{1},\varepsilon/2)}|g(x)-f(x)|=\epsilon/2<\epsilon,$$	
showing that the set of $G$-differentiability of $\varphi(x)=\sup_{t \in T}|x(t)|$ is dense in $C(T)$.	
\end{proof}

\begin{rem}
If we take all the points in which the functions $x \in C(T)$ are $F$-differentiable, this is equivalent to take all the isolated points in $T$. But by its  nature this set of isolated points is not dense at $T$, so by Proposition 2.2.2 of Fabian~\cite{fa3} we conclude that the set of $x\in C(T)$ such that $\varphi(x)$ is $G$-differentiable is not dense in $C(T)$.
\end{rem}

\section[$G$-derivative of norms in non-separable spaces]{$G$-derivative of norms in non-separable spaces}

In this section we present the main results, which show sets of $G$-differentiability of norms on non-separable Banach spaces. But first we will point out that the impossibility of extend  Phelps theorem to any non-separable Banach space. Indeed, in the case of space $\ell^{\infty} (\mathbb{R})$, Larman in~\cite {Larman} shows that the function $\varphi(x)=\limsup_{n\in\mathbb{N}}|x_ {n}|$ is  continuous and not $G$-differentiable at any point. This result also shows that even the condition of being Lipschitz is not enough to guarantee the consequences of Phelps Theorem in the non-separable case. In the following theorem the  Lipschitz condition is satisfied and the consequences of  Phelps Theorem are also satisfied.

\begin{thm}\label{l*1}
	Let $ E=\ell^{\infty}(\mathbb{R})$ be the Banach space of the bounded  sequences endowed with the norm $\varphi(x)=\sup_{n}|x_ {n}|$. If $B\subseteq\ell^{\infty}(\mathbb{R})$ is the set of the sequences $x\in \ell^{\infty}(\mathbb{R})$ for which there exist $p\in\mathbb{N}$ and $\epsilon>0$ such that $|x_{k}|<|x_{p}|-\epsilon$ for all $k\neq p$. Then
	\begin{enumerate}
	 \item The $G$-differentiability set of $\varphi$ is $B$ and the $G$-derivative of $\varphi$ in $x\in B$ is $$\partial_{G}\varphi(x,h) = sig (x_{p})h_{p}$$
	 \item $B$ is an open set and is dense.
	 \item There is a Gaussian measure $\mu$, such that $\mu(B^{c})=0$.
	\end{enumerate}
	\end{thm}
\begin{proof}	
${(1)}$ If $x\in B$, $\epsilon>0$ and $p\in\mathbb{N}$ satisfy that $|x_{k}|<|x_{p}|-\epsilon$ for all $ k \neq p $, $\delta = \epsilon / 2$ and if $\|h\|<\delta$, getting $$\varphi(x+h)=\sup_{n\in \mathbb{N}}|x_{n}+h_{h}|=|x_{p}+h_{p}|\quad \textnormal{y} \quad sig(x_{p}+h_{p})=sig(x_{p}),$$
with $u(h)=sig(x_{p})h_{p}\in\mathcal{L}(E,\mathbb{R})$. Then 
	\begin{eqnarray}
	r(x,h)&=&\varphi(x+h)-\varphi(x)-sig(x_{p})h_{p}=|x_{p}+h_{p}|-|x_{p}|-sig(x_{p})h_{p}\nonumber\\&=&\nonumber sig(x_{p})(x_{p}+h_{p}-x_{p}-h_{p})=0,\nonumber
	\end{eqnarray}
getting ${\displaystyle\lim_{h\rightarrow 0}}\frac{\|r(x, h)\|}{\|h\|}=0$. Then $\varphi$ is $F$-differentiable, which implies that $$u(h)=sig(x_{p})h_ {p}\  \textrm{is $G$-derivated for $x\in B$}.$$
	
	To show that $\varphi$ is not $G$-differentiable in $B^{c}$. Let $x \in B^{c}$ and $\beta=\sup_ {n \in \mathbb{N}}|x_{n}| $. There is a subsequence $ (x_ {k_ {n}})_{n\in\mathbb{N}}$ such that $x_{k_{n}}>0$ for all $n$ with $x_{k_{n }}\rightarrow\beta$ or a subsequence $(x_{k_{n}})_{n\in\mathbb{N}}$ such that $x_{k_{n}}<0$ for all $n$ with $-x_{k_{n}} \rightarrow \beta$ (in either case we proceed in the same way). Let $h=(h_{n})_{n\in\mathbb{N}} $, where $h_{j}=0$ if $j\neq k_{n}$, $h_{j}=1$ if $j=k_{n}$, $n$ is even number and $h_{j}=-1$ if $j=k_{n}$, $n$ is odd. Now, if $x_{k_ {n}}>0$ for all $n\in\mathbb{N}$ and $t$ is small enough, we have:$$\sup\{|x_{j}+th_{j}|:j\neq k_{n}\}\leq\beta,$$
	$$\sup\{|x_{j}+th_{j}|:j= k_{n}\textrm{ and $n$ is even}\}=\beta+t,$$  $$\sup\{|x_{j}+th_{j}|:j= k_{n}\textrm{ and $n$ is odd}\}=\beta-t.$$
	Then $\varphi(x+th)=\beta+t$ if $t>0$ and $\varphi(x+th)=\beta-t$ for $t<0$, hence
	$$d_{+}\varphi(x,h)=1\quad \textnormal{y}\quad d_{-}\varphi(x,h)=-1,$$
	therefore, $\varphi$ is not $G$-differentiable in $B^{c}$.
	
\medskip\noindent{(2)} If we take $x\in B$ then there exists $p \in \mathbb{N}$ and $\epsilon>0$ such that $|x_{k}|<|x_{p}|-\epsilon$ for all $k\neq p$. If we assume $y\in B(x,\epsilon/ 4)$ and $k \neq p$, we have
	\begin{eqnarray}
	|y_{k}|&\leq&|y_{k}-x_{k}|+|x_{k}|\leq \epsilon/4+|x_{p}|-\epsilon\nonumber\\&\leq& |x_{p}-y_{p}|+|y_{p}|-3\epsilon/4\nonumber\\&\leq& |y_{p}|-\epsilon/2\nonumber
	\end{eqnarray}
Therefore $y\in B$, showing that $B$ is open. To show that $B$ is dense in $\ell^{\infty}(\mathbb{R})$, take $x\in\ell^{\infty}(\mathbb{R})$ and $\epsilon>0$, therefore there is an index $p$ such that $\varphi(x)-\epsilon/4<|x_{p}|\leq\varphi(x)$. If we define $y=(y_{n})_{n\in\mathbb{N}}$ as $y_ {p} = sig(x_{p})(\varphi(x)+\epsilon/2)$ and $y_{k}= x_{k}$ if $k\neq p$. Then $ |y_{k}|\leq\varphi(x)=|y_{p}|-\epsilon/2$ if $k\neq p$ and therefore $ y \in B$. Moreover  	
\begin{eqnarray}
|y_{p}-x_{p}|&=&|sig(x_{p})(\varphi(x)+\epsilon/2)-sig(x_{p})|x_{p}|)|\nonumber\\&=&|\varphi(x)+\epsilon/2-|x_{p}||\nonumber\\&<&3\epsilon/4<\epsilon 
\nonumber\end{eqnarray}
Therefore, $\varphi(y-x)=|y_{p}-x_{p}|<\epsilon$.
		
\medskip\noindent{(3)} In $\ell^{\infty}(\mathbb{R})$ there is a Gaussian measure $\mu$, since in Section 2.4.3 of the Vakhania book, it is shown that there is a Gaussian $\mu$ measure about $\mathbb{R}^{\mathbb{N}}$ which satisfies $\mu(\ell^{\infty}(\mathbb{R}))=1$ if and only if it satisfies that for $a_{j}=E(x_{j})$ and $s_{kk}=E(x_{k}-a_{k})^{2}$  $$\{a_{j}\}\in\ell^{\infty}(\mathbb{R})\ \text{and}\ \sum_{k=1}^{\infty}e^{-r/s_{kk}}<\infty,\ \text{for some}\ r>0.$$
	On the other hand, given $t \in \mathbb{N}$. If it is defined $f_{t}:\ell^{\infty}(\mathbb{R})\rightarrow\mathbb{R}^{t}$, how	
	$$f_{t}(x)=(x_{1}, x_{2},\cdots,x_{t-1},x_{p}),\quad x\in\ell(\mathbb{R}),$$
$|x_{k}|<|x_{p}|-\epsilon_{p}$, for $k=1,\cdots,t-1$ and $\epsilon_{p}>0$. The function $f$ is continuous since for $y\in B(x,\delta)$, $x,y\in\ell^{\infty}(\mathbb{R})$ we have
\begin{eqnarray}
	\|f_{t}(x)-f_{t}(y)\|_{\mathbb{R}^{t}}&=&\|(x_{1}-y_{1},\cdots,x_{t-1}-y_{t-1},x_{p}-y_{p})\|_{\mathbb{R}^{t}}\nonumber\\&=&|x_{1}-y_{1}|+\cdots+|x_{t-1}-y_{t-1}|+|x_{p}-y_{p}|\nonumber\\&\leq&t\sup_{n}|x_{n}-y_{n}|=t\|x-y\|_{\ell^{\infty}(\mathbb{R})}\nonumber
	\end{eqnarray}
and then for all $\epsilon>0 $ there is a $\delta=\frac{\epsilon}{t}$ such that $\|xy\|_{\ell^{\infty}(\mathbb{R})}<\delta$ implies $\|f_{t}(x)-f_{t}(y)\|_{\mathbb{R}^{t}}<\epsilon$.
	
From  all of the above, for $\mu_{t}$ a Gaussian measure in $\mathbb{R}^{t}$ and $\mu$ the Gaussian measure in $\ell^{\infty}(\mathbb{R})$, one obtains $\mu(A)=\mu_{t}(f_{t}(A))$. Therefore, for $g_ {t}^{w}:\ell^{\infty}(\mathbb{R})\rightarrow\mathbb{R}^{t}$ such that 
	$$g_{t}^{w}:\ell^{\infty}(\mathbb{R})\rightarrow \mathbb{R}^{t},\text{with $x_{p}$ in the $w$th coordinate of $\mathbb{R}^{t}$},$$ 
	also satisfies $\mu(A)=\mu_{t}(g_{t}^{w}(A))$, for all $w\leq t\in\mathbb{N}$.
	
	If we takes $t=2$ and we defines $B_{2}=\bigcup_{w = 1}^{2}\{g_{2}^{w}(x):|x_{k}|<|x_{p}| \epsilon_{p}, \ k \neq w\}$, it can be  can easily proved that
	$$B_{2}=\mathbb{R}^{2}-\{(x_{1},x_{2}):x_{2}=x_{1},\ x_{2}=-x_{1}\},$$
	of which, $\mu_{2}({B}{2})=1$ and therefore $B_{2}^{c}=\{(x_{1},x_{2}):x_{2}=x_{1},\ x_ {2} =-x_{1}\}$ such that $\mu_{2}(B_{2}^{c})=0$.
	
Now, if the previous result is true for $t = n$, that is $\mu(B_{n})=1$ for $$B_{n}=\bigcup_{w=1}^{n}\{g_{n}^{w}(x):|x_{k}|<|x_{p}|-\epsilon_{p},\ k\neq w\}.$$ 
	Thus  
	\begin{eqnarray}
	B_{n+1}&=&\bigcup_{w=1}^{n+1}\{g_{n+1}^{w}(x):|x_{k}|<|x_{p}|-\epsilon_{p},\ k\neq w\}\nonumber\\ &=&\bigcup_{w=1}^{n-1}\{g_{n+1}^{w}(x):|x_{k}|<|x_{p}|-\epsilon_{p},\ k\neq w\}\nonumber\\& & \cup \bigcup_{w=n}^{n+1}\{g_{n+1}^{w}(x):|x_{k}|<|x_{p}|-\epsilon_{p},\ k\neq w\}\nonumber\\&=& B_{n-1}\times [ \epsilon_{p}-|x_{p}|,|x_{p}|-\epsilon_{p}] ^{2}\cup [\epsilon_{p}-|x_{p}|,|x_{p}|-\epsilon_{p}]^{n-2} \times B_{2}\nonumber\\ &\supseteq& B_{n}\times [ \epsilon_{p}-|x_{p}|,|x_{p}|-\epsilon_{p}]\cup [\epsilon_{p}-|x_{p}|,|x_{p}|-\epsilon_{p}]^{n-2} \times B_{2}\nonumber\\ &\supseteq& B_{n}\times [ \epsilon_{p}-|x_{p}|,|x_{p}|-\epsilon_{p}],\nonumber
	\end{eqnarray}
then of course, for all $x_{p}\in\mathbb{R}$, it is satisfied  
	\begin{eqnarray}
	\mu_{n+1}(B_{n+1})&\geq& \mu_{n+1}(B_{n}\times [ \epsilon_{p}-|x_{p}|,|x_{p}|-\epsilon_{p}])\nonumber\\&=&\mu_{n}(B_{n})\times\mu_{1}( [ \epsilon_{p}-|x_{p}|,|x_{p}|-\epsilon_{p}])\nonumber\\&=& \mu_{1}( [ \epsilon_{p}-|x_{p}|,|x_{p}|-\epsilon_{p}]),\nonumber
	\end{eqnarray}
    what implies $\mu_{n+1}(B_{n + 1})=1$ and that $\mu_{n+1}(B_{n+1}^{c})=0$, of what which for all $n\geq 2$ it is satisfied that $\mu_{n}(B_{n} ^{c})=0$.
	
Finally, as $B^{c}\subset B_{n}^{c}\times\ell^{\infty-\{1,\cdots, n\}}(\mathbb{R})$, where $B_{n}^{c} \times \ell^{\infty-\{1,\cdots,n\}}(\mathbb{R})$ is the set of the sequences in $\ell^{\infty}(\mathbb{R})$ with the first $n$ components in $B_{n}^{c}$, it is true that
\begin{eqnarray}
\mu(B^{c})&\leq&\mu(B_{n}^{c}\times \ell^{\infty-\{1,\cdots,n\}}(\mathbb{R}))\nonumber\\&\leq&\mu(B_{n}^{c}\times \ell^{\infty}(\mathbb{R}))\leq\mu_{n}(B_{n}^{c})\times\mu(\ell^{\infty}(\mathbb{R}))\nonumber\\&\leq&\mu_{n}(B_{n}^{c}),\nonumber
\end{eqnarray} 
	concluding that $\mu(B^{c})\leq\mu_{n}(B_{n}^{c})$ for all $n\geq 2$, that is $\mu(B^{c})=0$.
\end{proof}

Given $NBV[a, b]$ the Banach space of the normal real functions of bounded variation endowed with the norm $\varphi(f)=| f |([a, b])$. In the following theorem, a result similar to Larman's result in ~\cite{Larman} will be shown in the case of a norm.

\begin{thm}
$\varphi=|f|([a,b])$ it is not $G$-differentiable at any point $f\in NBV[a,b]$.
	\end{thm}
	
\begin{proof}
Let $c,d \in [a, b]$ and $f \in NBV[a, b]$ not constant, if we define $c$ as the number such that $\lim_{x\rightarrow c} f(x)=m \leq f(y)$ for all $y\in [a, b]$, a $d$ as the one that meets $\lim_{x\rightarrow d} f(x)=M \geq f(y)$ for all $y\in [a, b]$ and define $h\in NBV[a, b]$ as
	\begin{center}
	\begin{tabular}{lll}
	$h(x)= \left\{ \begin{array}{lll}
	0 & \textrm{if $x\in[a,\frac{c+d}{2}]$}\\ 
	1 & \textrm{if $x\in(\frac{c+d}{2},b]$.}
	 \end{array} \right.$
	\end{tabular} 
	\end{center}
Then for $\frac{\varphi(f+th)-\varphi(f)}{t}$, with $c\neq d$ one obtains
	\begin{center}
	\begin{tabular}{lll}
	$\frac{\varphi(f+th)-\varphi(f)}{t}= \left\{ \begin{array}{lll}
	\ \ 1 & \textrm{if $c<d$ and $t>0$} \\
	-1 & \textrm{if $c<d$ and $t<0$} \\
	-1 & \textrm{if $c>d$ and $t>0$} \\
	\ \ 1 & \textrm{if $c>d$ and $t<0$} 
	 \end{array} \right.$
	\end{tabular} 
	\end{center}
If $c=d$ and $\lim_{x \rightarrow c^{+}} f (x)=\alpha$, $\lim_{x \rightarrow c^{-}}f(x)=\beta$ , then
	\begin{center}
	\begin{tabular}{lll}
	$\frac{\varphi(f+th)-\varphi(f)}{t}= \left\{ \begin{array}{lll}
	\ \ 1 & \textrm{if $\alpha>\beta$ and $t>0$} \\
	-1 & \textrm{if $\alpha>\beta$ and $t<0$} \\
	-1 & \textrm{if $\alpha<\beta$ and $t>0$} \\
	\ \ 1 & \textrm{if $\alpha<\beta$ and $t<0$}
	\end{array} \right.$
	\end{tabular} 
	\end{center}
In the case where $f\in NBV[a, b]$ is constant, if one takes the function $h$ as in the previous cases, one obtains
	\begin{center}
	\begin{tabular}{lll}
	$\frac{\varphi(f+th)-\varphi(f)}{t}= \left\{ \begin{array}{lll}
	\ \ 1 & \textrm{if $t>0$} \\
	-1 & \textrm{if $t<0$} 
	\end{array} \right.$
	\end{tabular} 
	\end{center}
Therefore, the norm $\varphi $ is not $G$-differentiable at any point $f\in NBV[a,b]$.
\end{proof}

Consider the Banach space $L^{\infty} (\mathbb{R})$ of the actual measurable and bounded functions with the norm $\varphi(f) = \inf \{M:|f(x)|\leq M\}$, demonstrating a generalization of what is done in $\ell^{\infty}(\mathbb{R})$.
		
	\begin{thm}
	Let $B$ be the set of the continuous functions $f \in L^{\infty} (\mathbb{R})$ that have an only point  $x_{0}\in\mathbb{R}$ where $f$ assumes the supremum. The $G$-differentiability set of $\varphi$ is $B$ and if $f\in B$, then $$\partial_{G}\varphi(f,h)=sig(f(x_{0}))h(x_{0}).$$
	\end{thm}

\begin{proof}	
{(1)} Suppose that $f$ reaches the supremum only in $x_{0}\in\mathbb{R}$, then for $\varphi(h)<\delta$, where $\delta=\epsilon/ 2$, for $|f(x)|<|f(x_{0})|-\epsilon$ you have to $\varphi(f+h)=\inf\{M:|f(x)+h(x)|\leq M\}=|f(x_{0})+h(x_{0})|$ and $sig(f(x_{0})+h(x_{0}))=sig(f(x_{0}))$. Thus,{\setlength\arraycolsep{0pt} \begin{eqnarray}
	&\varphi&(f+h)-\varphi(f)-sig(f(x_{0}))h(x_{0})\nonumber\\&=& |f(x_{0})+h(x_{0})|-|f(x_{0})|-sig(f(x_{0}))h(x_{0})\nonumber\\&=&sig(f(x_{0})+h(x_{0}))(f(x_{0})+h(x_{0}))-sig(f(x_{0}))(f(x_{0}))-sig(f(x_{0}))h(x_{0})\nonumber\\&=&sig(f(x_{0}))(f(x_{0})+h(x_{0}))-sig(f(x_{0}))(f(x_{0}))-sig(f(x_{0}))h(x_{0})\nonumber\\&=&sig(f(x_{0}))(f(x_{0})+h(x_{0})-f(x_{0})-h(x_{0}))=0,\nonumber 
	\end{eqnarray}}showing that $\varphi$ is $F$-differentiable over $B$, which immediately implies the $G$-differen\-tiability of $\varphi$ over $B$.\\
	
Now it will be shown that $\varphi$ is not differentiable on $B^{c}$. Let $f\in B^{c}$ such that $\varphi(f)=w$ and suppose that it has two maximum $x_{0}$ and $x_{1}$, such that $f(x_{0})=f(x_{1})=w$, $x_{1}>x_{2}$ and $\epsilon=|x_{1}-x_{0}|$. If you define $h\in L ^{\infty}(\mathbb{R})$ as
	\begin{center}
	\begin{tabular}{lll}
	$h(x)= \left\{ \begin{array}{lll}
	-1 & \textrm{if $x\in (-\infty, x_{0}+\epsilon/2)$} \\
	\ \ 1 & \textrm{if $x\in [x_{0}+\epsilon/2),\infty)$}.
	 \end{array} \right.$
	\end{tabular} 
	\end{center}
Then $$\inf\{M:|f(x)+th(x)|\leq M,\quad x\in (-\infty, x_{0}+\epsilon/2)\}=w-t$$ $$\inf\{M:|f(x)+th(x)|\leq M,\quad x\in [x_{0}+\epsilon/2,\infty)\}=w+t.$$
Therefore, $\varphi(f + th)=w-t$ if $t> 0$ and $\varphi(f + th)=w+t$ if $t<0$. Consequently
	\begin{center}
	\begin{tabular}{lll}
	$\varphi(f+th)-\varphi(f)= \left\{ \begin{array}{lll}
	-t & \textrm{if $t>0$} \\
	\ \ t & \textrm{if $t<0$}.
	 \end{array} \right.$
	\end{tabular} 
	\end{center}
Hence $\lim_{t\rightarrow 0^{+}} \frac{\varphi(f + th)-\varphi(f)}{t}=-t\neq t=\lim_{t \rightarrow 0^{-}}\frac{\varphi(f+th)-\varphi(f)}{t}$, showing the desired conclusion.
\end{proof}

\section[$G$-derivative of locally Lipschitz functions]{$G$-derivative of locally Lipschitz functions}

In Phelps Theorem, a locally Lipschitz function is taken in a separable Banach space $X$. Now, it will be shown that this condition can also be considered for non-separable spaces, if the Banach space is assumed to be $X$ as a projective limit of finite Banach spaces in which the connecting functions meet the condition of being $H$-differentiable. We will begin with a preliminary proposition. 

\begin{prop}\label{composicion}
	Given a function $f:\Omega\subseteq X \rightarrow \mathbb{R}$ locally Lipschitz and not constant over any open $\Omega$. If $f$ is the composition of  two  functions $h:Y\rightarrow \mathbb{R}$ and $p:\Omega\rightarrow Y$, such that $f(x)=h(p(x))$. Then  $h$ is locally Lipschitz.   
\end{prop}

\begin{proof}
Since $f$ is locally Lipschitz on $x\in\Omega$, then there exists $B(x,r)$, $k>0$ such that for $y\in B(x,r)$ is met:
$$|f(x)-f(y)|=|h(p(x))-h(p(y))|\leq k\|x-y\|_{X},$$       
if we assume $h$ is not Lipschitz, then for all $k_{0}>0$ and all $y\in B(x,r)$ you have:
$$|h(p(x))-h(p(y))|>k_{0}\|p(x)-p(y)\|_{Y},$$ 
then using the first inequality
$$\|p(x)-p(y)\|_{Y}<\frac{k}{k_{0}}\|x-y\|_{X}\leq \frac{k}{k_{0}}r,$$
this for all $k_{0}>0$, from which it is concluded that $ p(x)=p(y)$ and in turn that $f(x)=f(y)$ for all $y\in B(x,r)\subseteq\Omega$, which contradicts the hypothesis in which $f$ is taken as a non-constant function over any open $\Omega$.
\end{proof}

Now some definitions that will be required are presented.  

\begin{defn} Let $(T;\preccurlyeq)$ be a directed set and $s,t\in T$
	\begin{enumerate}
		\item The family $(X_{t})_{t\in T}$ of topological spaces is {\it projective} if for each pair $(s,t)$ in $T\times T$ with $s\preccurlyeq t$ there exists a continuous function $p_ {st}:X_ {t}\rightarrow X_{s}$ ({\it connecting function}) such that $p_{st}\circ p_{tw}=p_{sw}$ if $s\preccurlyeq t\preccurlyeq w$.
		
		\item {\it The projective limit} $X_{T}$ of a projective family $(X_{t})_{t\in T}$ of finite dimension topological spaces, is the set of $x=(x_{t})_{t \in T}\in \prod_{t\in T}X_{t}=X_{T}$ such that there is a connecting function $p_{st}(\cdot)$ that meets $x_{s}=p_{st}(x_{t})$ for all $s\preceq t$, in particular to the connecting function $p_{s}(\cdot):X_{T}\rightarrow X_{s}$ is called {\it sth projection}.
		
		\item Given $(X_{t})_{t\in T}$ a projective family, $C_{T}$ the set of all continuous functions of $X_{T}$ in $\mathbb{R}$ and $C_{t}$ the set of all continuous functions of $X_{t}$ in $\mathbb{R}$. {\it A function $f\in C_{T}$ is cylindrical} if there exists a $t\in T$ and $h \in C_{t}$ such that $f=h\circ p_{t}$.
	\end{enumerate}
\end{defn}

\begin{ex}
	If $T$ is the set of all the subspaces of finite dimension of a topological space $X$ ordered by inclusion. For each $F\in T$, the projection of $X$ over $F$ will be denoted by $p_{F}$. If $F,G\in T$ and $F\subseteq G$ It will be denoted by $p_{FG}$ to the orthogonal projection of $G$ over $F$, It is immediate that if $F\subseteq G$ then $p_{F}=p_{FG}\circ p_{G}$. Therefore $(F)_{F\in T}$ is a projective family of topological spaces called {\it projective family of the subspaces of finite dimension}. Now, for every $F\in T$ and every $\mu_{F}$ a Borelian measure in $F$, it will be said that a family $(\mu_{F})_{F\in T}$ is projective if for every $F\subseteq G$ it is satisfied that $\mu_{F}=p_{FG}(\mu_{G})$. In addition, every measure in $bor(X)$ generates a projective family $(\mu_{F})_{F\in\mathcal{F}}$ of borelian measures, where $\mu_{F}=p_{F}(\mu)$ and a function $f:X\rightarrow\mathbb{R}$ is cylindrical if there is a finite dimension subspace $F$ and a function $f_{F}:F\rightarrow \mathbb{R}$ such that $f=f_{F}\circ p_{F}$. 
\end{ex}

\begin{rem}\label{composicion2}\ 
	\begin{enumerate}
		\item Given $F\subset X$ a subspace of finite dimension. A cylindrical function $f=f_{F}\circ p_{F}$ is a borelian function if and only if $f_{F}$ is a borelian function.
		
		\item Given $t\preceq s$ y $f\in Cil(X_{T})\equiv {C_{T}}$ a cylindrical function, it is true that for $m_{t}\in Cil(X_ {t})\equiv C_{t}$ and $m_{s}\in Cil(X_{s})\equiv C_{s}$ such that
		$$f=m_{t}\circ p_{t}=m_{s}\circ p_{s}$$
		
		\item If $f$, $m_{t}$, $p_{t}$ are $H$-differentiable, then for $f=m_{t} \circ p_{t}$, the following applies:
		$$d_{H}f(x,h)=d_{H}[m_{t}(p_{t})(x,h)]=d_{H}m_{t}(p_{t}(x,h))d_{H}p_{t}(x,h),$$
		therefore, the existence of $d_{H} m_{t}(x_{t},h)$ implies the existence of $d_{H}f(x,h)$. 
	\end{enumerate}
\end{rem}

In the following theorem, the extension of the Phepls Theorem is made in the case of cylindrical functions and locally Lipschitz defined on a projective limit $X_{T}$, which will be assumed not separable.

\begin{thm}\label{extension1}
Let $X_{T} $ be a non-separable Banach space, which is the projective limit of the family of finite dimensional spaces $(X_{t})_{t\in T} $ with $t$-th projections $p_{t}$ and $\Omega \subseteq X_{T}$ an open. For a cylindrical function $f:\Omega \rightarrow \mathbb{R}$, $ f (x) = m_{t}\circ p_{t}(x)$, not constant over any open $\Omega$, locally Lipschitz, $\mu$ and $\mu_{t}=p_{t}(\mu)$ Gaussian measures defined on each $X_{T}$ and on $X_{t}$ respectively, is satisfied:
$$\mu(\{x\in\Omega:d_{G} f(x,h)\ does\ not\ exist\}\cap\Omega)=0.$$
\end{thm}

\begin{proof}
Since $f(x)=m_{t}\circ p_{t}(x)$, by Proposition \ref{composicion} we get that $m_{t}$ is locally Lipschitz. Since $X_{t}$ is of finite dimension, $\mathbb{R}$ has the RNP and $m_{t}: X_{t}\rightarrow\mathbb{R}$, then from  Phelps Theorem applied to each $m_{t}$ it is true that $$\mu_{t}(\{x_{t}\in\Omega_{t}:d_{G}m_{t}(x_{t},h)\ does\ not\ exist \})=0,$$ where $p_{t}(\Omega)=\Omega_{t}$, since $p_{t}(\cdot)$ is continuous over $\Omega$ and therefore measurable. Hence
{\setlength\arraycolsep{0pt}	 
\begin{eqnarray}
\mu&(&\{x\in\Omega: d_{G}f(x,h)\ does\ not\ exist\})\nonumber\\&=&\mu(p_{t}^{-1}(\{x_{t}\in\Omega_{t}: d_{G}m_{t}(x_{t},h)\ does\ not\ exist\}))\nonumber\\&=&\mu_{t}(\{x_{t}\in\Omega_{t}: d_{G}m_{t}(x_{t},h)\ does\ not\ exist\}).\nonumber
\end{eqnarray}}Then $\mu(\{x\in\Omega:d_{G}f(x,h)\ does\ not\ exist\})=0$.
\end{proof} 

\begin{rem} If we define $f=m_{t}\circ p_{t}(x)$ over $C_{T}(Z) $ ($C_ {T}(Z)$ the set of continuous functions of $X_{T}$ in $Z$), $m_{t}$ over $C_{t}(Z)$ ($C_{t}(Z)$ the set of continuous functions of $X_{t}$ in $Z$), for $Z$ a Banach space with the RNP, if the demonstration of the Proposition~\ref{composicion} from $\mathbb{R}$ to $Z$ is extended and if the remaining conditions of the previous theorem are assumed. Then we can demonstrate in the same way to the above that $$\mu(\{x\in\Omega: \partial_{G} f(x, h) \ does\ not\ exist\}) = 0.$$
\end{rem}

In the following example we will show a Lipschitz function defined on a Banach space, which satisfies the hypothesis of the previous theorem.

\begin{ex}\label{ex1}
Since $\ell^{\infty}(\mathbb{R})$, with the norm $\|x\|_ {\infty}=\sup_{k\in\mathbb{N}}|x_{k}|$, can be defined as the projective limit of finite spaces $(\mathbb{R}^{n})_ {n\in\mathbb{N}}$. If we define $f:\ell^{\infty}(\mathbb{R})\rightarrow[0,\infty)$, as the convergent series
$$f(x)=f(x_{1},\cdots,x_{n},\cdots)=\sum_{k=1}^{\infty}\frac{|x_{k}|}{k^{2}},$$
this function is continuous, just show that it is Lipshitz in $\ell ^{\infty}(\mathbb {R})$ and in effect, for $x, y \in \ell^{\infty} (\mathbb{R})$ is satisfied that	
$$\sum_{k=1}^{\infty}\frac{|x_{k}-y_{k}|}{k^{2}}\leq\sum_{k=1}^{\infty}\frac{|x_{k}|}{k^{2}}+\sum_{k=1}^{\infty}\frac{|y_{k}|}{k^{2}},$$
thus $\sum_{k = 1}^{\infty}\frac{|x_{k}-y_{k}|}{k^{2}}$ converge and as $|x_{k}-y_{k}| \leq \|xy\|_{\infty}$ for all $k\in \mathbb{N}$, then
\begin{eqnarray}
|f(x)-f(y)|&=&\left|\sum_{k=1}^{\infty}\frac{|x_{k}|-|y_{k}|}{k^{2}}\right|\leq\sum_{k=1}^{\infty}\frac{|x_{k}-y_{k}|}{k^{2}}\leq\sum_{k=1}^{\infty}\frac{\|x-y\|_{\infty}}{k^{2}}\nonumber\\&\leq&\|x-y\|_{\infty}\sum_{k=1}^{\infty}\frac{1}{k^{2}}\leq\|x-y\|_{\infty}.\nonumber
\end{eqnarray}
Therefore, there exists a $k_{0}=1$ such that for all $x,y\in\ell^{\infty}(\mathbb{R})$, it is satisfied $$|f(x)-f(y)|\leq k_{0}\|x-y\|_{\infty}.$$ 

On the other hand, for $x\in\ell^{\infty}(\mathbb{R})$ we can say that $p_{t}:\ell^{\infty}(\mathbb{R})\rightarrow\mathbb{R}^{t}$ defined as $p_{t}(x)=(x_{1},\cdots,x_{t})$ and $m_{t}:\mathbb{R}^{t} \rightarrow \mathbb{R}$ defined as $m_{t}(x_{1},\cdots,x_{t})=\sum_{k=1}^{t}\frac{|x_{k}|}{k^{2}}$ are continuous functions, since they are respectively satisfied:
$$\sum_{k=1}^{t}|x_{k}-y_{k}|\leq t\|x-y\|_{\infty} \quad y \quad \sum_{k=1}^{t}\frac{|x_{k}-y_{k}|}{k^{2}}\leq \sum_{k=1}^{t}|x_{k}-y_{k}|$$
that is, $p_{t}$ and $m_{t}$ are Lipschitz functions.
	
Also, if $x=(x_{k})_{k\in\mathbb{N}}\neq 0_{\ell^{\infty}(\mathbb{R})}$ for $t\neq 0$ small enough that $sig(x_{k}+ th_{k})=sig(x_{k})$, we obtain: $$\frac{f(x+th)-f(x)}{t}=\sum_{k=1}^{\infty}\left(\frac{|x_{k}+th_{k}|-|x_{k}|}{tk^{2}}\right)=\sum_{k=1}^{\infty}sig(x_{k})\frac{h_{k}}{k^{2}}.$$  
    Hence it is clear that for all $h\in\ell^{\infty}(\mathbb{R})$ 
	$$d_{G}f(x,h)=\sum_{k=1}^{\infty}sig(x_{k})\frac{h_{k}}{k^{2}}.$$
	Now, if $x=0_{\ell^{\infty}(\mathbb{R})}$, then 
	$$\frac{f(0+th)-f(0)}{t}=\frac{|t|}{t}\sum_{k=1}^{\infty}\frac{|h_{k}|}{k^{2}}.$$
	Then there is not $d_{G} f(0,h)$. Finally, if one defines $m_{t}(x)=\sum_{k=1}^{t}\frac{|x_{k}|}{k^{2}}$, one obtains $$d_{G}m_{t}(x,h)=\sum_{k=1}^{t}sig(x_{k})\frac{h_{k}}{k^{2}}$$ 
	and as in the previous case $d_{G} m_{t}(x,h)$ does not exist in $(0_{1},\cdots, 0_{t})$.
\end{ex}

Now, some extensions of the previous theorem are proposed, in which different arrival spaces are established to $\mathbb{R}$.

\begin{prop}
		Under the  same assumptions of Theorem~\ref{extension1}, adding the continuity of $f$  and taking $g$ a Lipschitz function of $\mathbb{R}$ in $\mathbb{R}^{n}$. Then for $\mathcal{H}=g\circ f$, $f(x)=m_{t}\circ p_{t}(x)$, $\mathcal{H}$ is $G$-differentiable except in a  null Gaussian set  with respect to the open $\Omega$ in the aforementioned theorem. 
	\end{prop}
	
    \begin{proof} 
    	Using the Theorem~\ref{extension1}, for $\mu$ a Gaussian measure on $X_{T}$, it is true that $\mu(\{x\in\Omega:d_{H}f(x,h)\ does\ not\ exist\})=0$. On the other hand, from Rademacher Theorem it is true that $d_{H}g(t)$, $t\in\mathbb{R}$, exists except in a zero Gaussian measure set, in this case Lebesgue null of $\mathbb{R}$, therefore
	$$d_{H}\mathcal{H}(x,h)=d_{H}g(f(x),\tilde{h})\cdot d_{H}f(x,h),\quad \tilde{h}\in\mathbb{R}, $$  
	which,
	{\setlength\arraycolsep{0pt}
		\begin{eqnarray}
		&\mu&(\{x\in \Omega: d_{H}\mathcal{H}(x,h)\ does\ not\ exist\})\nonumber\\&=&\mu(\{x\in \Omega: d_{H}g(f(x),\tilde{h})\ does\ not\ exist\}\cup\{x\in \Omega: d_{H}f(x,h)\ does\ not\ exist\})\nonumber\\&\leq& \mu(\{x\in \Omega: d_{H}g(f(x),\tilde{h})\ does\ not\ exist\})+\mu(\{x\in \Omega: d_{H}f(x,h)\ does\ not\ exist\})\nonumber\\&=& \mu(\{x\in \Omega: d_{H}g(f(x),\tilde{h})\ does\ not\ exist\})+0\nonumber\\&=& \lambda(\{t\in \mathbb{R}: d_{H}g(t,\tilde{h})\ does\ not\ exist\})=0,\nonumber
		\end{eqnarray}}
	where the last equality is due to the continuity of $f$.
    \end{proof}

	\begin{prop}
		Under the same assumptions of Theorem~\ref{extension1}, adding $f$ continuous and taking $g$ a local Lipschitz function of $\mathbb{R}$ in a $Y$ space with the RNP. Then, for $\mathcal{H}=g\circ f$, $f(x)=m_{t}\circ p_{t}(x)$, $\mathcal{H}$ is $G$-differentiable except in a set of null Gaussian measure, with respect to the open $\Omega$ of the aforementioned theorem.
	\end{prop}
	
	\begin{proof}
		Using the Theorem~\ref{extension1}, Phelps Theorem and applying similar argument to the previous proposition, the result is obtained.
	\end{proof}	 
	
	\begin{cor}
		Assuming the same hypotheses of the previous proposition, taking $Y=X'$ the dual space of a Banach space $X$ and if $\psi:X'\rightarrow X$ is $H$-differentiable. Then $X$ is Asplund and $\tau=\psi\circ\mathcal{H}: \Omega\subseteq X_{T} \rightarrow X$ is $H$-differentiable except for a null Gaussian set.   
	\end{cor}
	
	\begin{proof}
		Since $\mathcal{H}=g\circ f$, for $f$ and $g$ locally Lipschitz, then $\mathcal{H}:\Omega\rightarrow X'$ is locally Lipschitz. Therefore, from what was done in the previous proposition and what is done in the Theorem~\ref{composicion00}, we conclude that $\tau=\psi\circ\mathcal{H}$ is $H$-differentiable except in a  set of null Gaussian measure.
	\end{proof}



\end{document}